\documentclass[11pt]{article}
\usepackage[cp1251]{inputenc}
\usepackage[russian,english]{babel}
\usepackage{amssymb}
\usepackage{amsthm}
\usepackage{graphicx}

\newtheorem{theorem}{Theorem}[section]

\newtheorem{definition}{Definition}[section]
\newtheorem{remark}{Remark}[section]

\begin{document}
\begin{center}
\textbf{V.~I.~ Korobov, T.~V.~Revina\\On robust feedback for systems with multidimensional control}\\
\end{center}
Department of Applied Mathematics, School of Mathematics and
Computer Science,\\
V. N. Karazin Kharkiv National University,\\
E-mail address: vkorobov@univer.kharkov.ua, t.revina@karazin.ua\\

The paper deals with local robust feedback synthesis for systems
with multidimensional control and unknown bounded perturbations.
Using V.~I.~Korobov's controllability function method, we
construct a bounded control which steers an arbitrary initial
point to the origin in some finite time;  an estimate from above
for the time of motion is given. We have found the range of a
segment where the perturbations can vary. As an example we
consider the problem of stopping the oscillations of the system of
two coupled pendulums. \vskip2mm

\smallskip


\begin{center}
\section{Introduction and Problem Statement}
\end{center}

The paper deals with the synthesis problem, i.e. the problem of
constructing a control which depends on phase coordinates and
steers an arbitrary initial point from some neighborhood of the
origin to the origin in some finite time. Besides the control
should satisfy some preassigned constrains. In  \cite{Kor2}
methods for solving the feedback synthesis problem for a linear
system are given. Further we consider the synthesis problem for
the linear system with continuous \textit{bounded unknown
perturbations}. In the present paper we find such constraint for
the unknown perturbations that the control which solves the
synthesis problem for the system without the perturbation also
solves the synthesis problem for the perturbed system.

For the first time, the concept of the feedback synthesis has been
introduced and investigated in paper \cite{korobov_1} written in
Russian. In the English translation of this paper and in other
papers of its author,  this concept has been literally translated
from Russian as ''positional synthesis''. Later the concept has
been introduced and studied in \cite{Desoer,Kravaris} wherein it
has been called ``feedback synthesis''. Now, the term ''feedback
synthesis'' is generally used for the concept of the synthesis
introduced in \cite{korobov_1}. The controllability function
method is introduced in \cite{korobov_1}. In this method the angle
between the direction of motion and the direction of decrease of
the controllability function is not less than the corresponding
angle in the dynamic programming method, and no more than in a
method of Lyapunov function \cite[p. 10]{Kor2}. The main advance
of the controllability function method is finiteness of the motion
time. Among other authors developing such approach we would like
to mention \cite{Weiss}. Herein the concept of finite time
stability involves the bounding of trajectories within specifical
domains of the state space during a given finite time interval. A
bit later, the problem of steering an arbitrary initial point from
some neighborhood of the origin to the origin (or in general case
in equilibrium point) in a finite time has been called
''finite-time stabilization'' (see, e.g., \cite{Ryan,Bhat}). In
contrast to this problem the controllability function method is
solve  the problem of steering an arbitrary initial point to
generally non-equilibrium point in a finite time. The paper
\cite{non-equilibrium} is devoted to the problem of construction
of a constrained control, which transfers a control system from
any point to a given non-equilibrium point in a finite time in
global sense.

Let us consider the system
\begin{equation}\label{f_1} \dot{x}=(A_0+K+R(t,x))x+B_0u,
\end{equation}
 where $t\geq 0,\;$ $x\in Q\subset\mathbb{R}^n,$ $Q$ is
a neighborhood of the origin; $u\in \mathbb{R}^r$ is a control
satisfying the constraint $\|u\|\leq 1;$
 $ A_0$ is $ (n\times n)$ matrix of the form
 $A_0=\textrm{diag} \left(A_{01}, \ldots, A_{0r}\right),$
where $A_{0i}$ are $(n_i\times n_i)$ matrices of the form
$A_{0i}=\left(\begin{array}{cccccc}
0&1&0&\ldots&0&0\\
0&0&1&\ldots&0&0\\
0&0&0&\ldots&0&1\\
0&0&0&\ldots&0&0
\end{array}\right),$  $i=1,\ldots,r;\;$ $n_1\ge n_2\ge\ldots\ge n_r\ge 1,\;$
$n_1+\ldots+n_r{=}n;\;$ $B_0$ is a $(n\times r)$ matrix whose
elements $(B_0)_{s_ii}$ are equal to  1, $s_i=n_1{+}\ldots{+}
n_i,\;$ $i=1,\ldots,r$ and the others are equal to zero; the
elements of matrix $K$ which are in row $s_i$ (in other words a
row which contains a control) are equal to $k_{s_ij},$ and the
other elements are equal to zero,
$R(t,x)=\textrm{diag}\left(R_1(t,x), \ldots,
R_r(t,x)\right)+\hat{R}(t,x),\;$ $R_i(t,x)=$
\begin{equation}\label{newr}
=\left(\begin{array}{ccccccc}
r_{(s_{i-1}+1)1}&r_{(s_{i-1}+1)2}&0&0&\ldots&0&0\\r_{(s_{i-1}+2)1}&r_{(s_{i-1}+2)2}&r_{(s_{i-1}+2)3}&0&\ldots&0&0\\
&&&\ldots&&&\\
r_{(s_i-2)1}&r_{(s_i-2)2}&r_{(s_i-2)3}&r_{(s_i-2)4}&\ldots&r_{(s_i-2)(s_i-1)}&0\\
r_{(s_i-1)1}&r_{(s_i-1)2}&r_{(s_i-1)3}&r_{(s_i-1)4}&\ldots
&r_{(s_i-1)(s_i-1)}&r_{(s_i-1)s_i}\\
r_{s_i1}&r_{s_i2}&r_{s_i3}&r_{s_i4}&\ldots&r_{s_i(s_i-1)}&r_{s_is_i}
\end{array}\right),
\end{equation}
the elements of matrix $\hat{R}(t,x)$ which are in row $s_i$ (in
other words, a row which contains a control) are equal to
$r_{s_ij},$ and the other elements are equal to zero,
$r_{mj}=r_{mj}(t,x).\;$ We assume that functions $r_{mj}(t,x)$ are
\textit{unknown}, and we call such systems \textit{robust
systems}, see for ex. \cite[p. 173]{Pol}. We assume that the
functions $r_{mj}(t,x)$ satisfy an imposed constraints
\begin{equation}\label{const}
\max\limits_{1\leq j\leq m+1\leq
n_i,\;i=1,\ldots,r}|r_{mj}(t,x)|\leq \Delta.
\end{equation}
It is necessary to to find $\Delta$ and to construct \textit{a
bounded control} which steers an arbitrary initial point $x_0\in
Q$ to the origin \textit{in a finite time} for any perturbation
matrix $R(t,x)$ under condition (\ref{const}).

As a classical example of problem of this kind, we can mention the
problem of control over the motion of a cart over the surface with
an unknown bounded friction. The process of motion of this system
is described by the following equations
$$\left\{\begin{array}{l} \dot{x}_1=x_2,\\\dot{x}_2=r_{22}(t,x_1,x_2)x_2+u.\end{array}\right.$$
The term $r_{22}(t,x_1,x_2)x_2$ is sliding frictional force and
$r_{22}(t,x_1,x_2)$ is the coefficient of the nonlinear viscous
friction which is an unknown function and satisfies the constraint
$ |r_{22}(t,x_1,x_2)|\leq \Delta.$ The constraint under
consideration on $r_{22}(t,x_1,x_2)$ allow a "negative" friction.

The general approach to admissible control synthesis problem for
an arbitrary nonlinear autonomous control system has been given by
V.~I.~Korobov in \cite{korobov_1}. In the same paper an estimate
for the time of motion (settling-time function) from an arbitrary
initial point to the origin has been given. Recently, the problem
of finite-time stabilization has been formulated in several
different ways \cite{Kor2},\cite{Pol}-\cite{Su}. The article
\cite{Kor_Sk} describes a method for solving the feedback
synthesis problem for systems with multidimensional control and
without
 perturbations (i. e. $R(t,x){\equiv}0$). Moreover, in
this case the controllability function is the time of motion. In
\cite{kor_rev}, we have solved the robust synthesis problem for a
case with one perturbation and a scalar control. In \cite{umg},
the case when $R(t,x)=p(t,x)R,\;$ $K\equiv 0$ and the control is
scalar has been considered.

In \cite{cai}, an adaptive fuzzy finite-time control scheme has
been proposed for a class of nonlinear systems with unknown
nonlinearities. The proposed scheme can guarantee that states of
the closed-loop system converge to a small neighborhood of the
origin in finite time. The book \cite[p. 201]{Pol} deals with the
problem of robust stabilization for systems with constant affine
perturbations. In \cite{Polyakov}, the Lyapunov function method
has been suggested to analyze the finite-time stabilization of the
system $\dot{x}(t)=A_0x+B_0u(t)+d(t,x(t)),$ where $u(t)$ is a
scalar function and $d(t,x)$ is measurable and uniformly bounded
in the variable $t$ function. In \cite{Ovseevich,Polyakov}, the
finite-time stabilization conditions have been formulated in the
form of linear matrix inequalities. In \cite{Su}, the problem of
finite-time stabilization for the second order system of general
form (or double integrator) with a scalar control has been
considered.

First of all, we describe the conditions which the perturbations
$r_{mj}(t,x)$ must satisfy.

\begin{definition} By set of admissible perturbations $\mathcal{R}$ we denote a set of matrices $R(t,x)$ whose
elements are functions $r_{mj}(t,x):\;[0,+\infty)\times
Q\rightarrow\mathbb{R}$ such that the following conditions are
satisfied:

i) $r_{mj}(t,x)$ are continuous in variables $t$ and $x;$

ii) $\max\limits_{1\leq j\leq m+1\leq
n_i,\;i=1,\ldots,r}|r_{mj}(t,x)|\leq \Delta$ for all
$(t,x)\in[0,+\infty)\times Q;$

iii) in each domain $\overline{K}_1(\rho_2)=\{(t,x):\; 0\leq t<
+\infty,\;\|x\|\leq \rho_2\},\;$ the vector function $R(t,x)x$
satisfies the Lipschitz condition $$|R(t,x'')x''-R(t,x')x'|\leq
\ell_1(\rho_2)\|x''-x'\|.$$
\end{definition}

If $R(t,x)\equiv 0$, then (\ref{f_1}) is canonical system:
$\dot{x}=(A_0+K)x+B_0u.$ This concept has been introduced in
\cite{korobov_1} for the first time. Also  this system has been
called "chain of integrators system" (for second order system see
for ex. \cite{Bhat}). In the analyzed approach, this system plays
the key role because the solution of the synthesis problem for an
arbitrary linear system with a multidimensional control can be
reduced to the solution of the synthesis problem for the canonical
system \cite[p.~105]{Kor2}. The canonical system is completely
controllable. In \cite[Theorem 2.3]{Kor2};\cite{Kor_Sk} the
control $u(x)$ which solves the synthesis problem for the
canonical system is given.

\begin{definition}
The problem of finding such range of perturbations $r$ that the
trajectory $x(t)$ of the closed-loop system with the control
$u(x)$
\begin{equation}\label{poln}
\dot{x}=(A_0+K+R(t,x))x+B_0u(x),
\end{equation}
starting at an arbitrary initial point $x(0)=x_0\in Q,$ ends at
the origin at some finite time $T(x_0,\mathcal{R}),$ i. e.
$\lim\limits_{t\rightarrow T(x_0,\mathcal{R})}x(t)=0,$ is said to
be the local robust feedback synthesis.  If $Q=\mathbb{R}^n$, this
problem is called the global robust feedback synthesis.
\end{definition}

Obviously, if $r_{11}(t,x)\equiv0$ and $r_{12}(t,x)\equiv-1$  then
the first coordinate $x_1$ in (\ref{f_1}) is uncontrollable;
therefore, the problem will not be solvable for any value of
$\Delta.$

The paper is organized as follows. In Section 2, some basic
concepts of the  controllability function method are given.
Section 3 represents the main results. In Section 3.3 we consider
the problem of stopping the oscillations of the system of two
coupled pendulums.

\begin{center}
\section{Background: the Controllability Function Method}
\end{center}

In this Section we recall some basic concepts and some results of
the controllability function method \cite{Kor2,korobov_1}. Let us
consider a nonlinear system of the form
\begin{equation}\label{k5}\dot{x} = f(x,u),
\end{equation}
where  $x\in Q\subset \mathbb{R}^n$ and $u\in \Omega \subset
\mathbb{R}^r,$ moreover, $\Omega$ is such that $ 0\in
\;int\;\Omega,\;f(0,0){=}0.$

\begin{definition}
The problem of constructing a control of the form $u=u(x),$
$\;x\in Q$  is said to be the local feedback synthesis if:\\i) $
u(x)\in {\rm \Omega};\;$\\ii) the trajectory $x(t)$ of the
closed-loop system $\dot{x}=f(x,u(x)),$ starting at an arbitrary
initial point $x_0\in Q,$ ends at the origin at some finite time
$T(x_0).$ If $Q=\mathbb{R}^n,$ the problem is called the global
feedback synthesis.
\end{definition}

The sufficient conditions for solvability the feedback synthesis
for system (\ref{k5}) were formulated in \cite[Theorem 1.1]{Kor2}.

Let us describe one of possible approaches to the solution of the
feedback synthesis  for the canonical system \cite[Theorem
2.3]{Kor2};\cite{Kor_Sk}:
\begin{equation}\label{kan}
\dot{x}=(A_0+K)x+B_0u,
\end{equation}
where $x\in\mathbb{R}^n,\;$ $u\in \mathbb{R}^r$ is a control which
satisfies the constraint $\|u\|\leq 1.$ It should be noted that
system (\ref{f_1}) coincides with the completely controllable
system (\ref{kan}) when $R(t,x)\equiv 0.$ Let us set
\begin{equation}\label{deff}
F^{-1} =\int\limits_0^1 (1-t)e^{-A_0t}B_0B_0^{*}e^{-A_0^{*}t}dt.
\end{equation}

Let $D(\Theta)$ be a diagonal matrix of the form
\begin{equation}\label{defd}
D(\Theta)=\textrm{diag}(D_1(\Theta),\ldots,D_r(\Theta)),\quad\mbox{where}\quad
D_i(\Theta)=\textrm{diag}\left(\Theta^{-\frac{2n_i-2j+1}{2}}\right)_{j=1}^{n_i}.
\end{equation}

\begin{theorem}\label{thkor}\cite[Theorem
2.3]{Kor2};\cite{Kor_Sk}.  The controllability function
$\Theta=\Theta(x)$ is defined for $ x\neq 0$ as a unique positive
solution of the equation
\begin{equation}\label{k10}
2a_0\Theta=(D(\Theta)FD(\Theta)x,x),\
\end{equation}
where the constant $a_0$ satisfies the inequality
\begin{equation}\label{k11}
0<a_0\leq
\frac{2}{\|F^{-1}\|\cdot(\|B_0^*F\|+2\max\{c^{n_1},c\}\|B_0^*K\|)^2},
\end{equation}
besides the domain of solvability synthesis problem  is ellipsoid
of the form\linebreak $\;Q=\{x:\;\Theta(x)\leq c\}.$ At $x=0$ we
put $\Theta(0)=0.$

Then at the domain $Q$  the control
\begin{equation}\label{k12}
u(x)=-\left(\frac{1}{2}\;B_0^*D(\Theta(x))FD(\Theta(x))+B_0^*K\right)x
\end{equation}
solves the local feedback synthesis for system (\ref{kan}) and
satisfies the constraint $\|u(x)\|\leq 1.$ Moreover, in this case
the equation $\dot\Theta(x)=-1$ holds, i. e. the controllability
function $\Theta(x)$ equals to the time of motion from any initial
point $x\in Q$ to the origin.

In the case when $K\equiv 0,$ the synthesis is global.
\end{theorem}

\begin{center}
\section{The Solution of the Robust Feedback Synthesis}
\end{center}

Let us consider system (\ref{f_1}).Eq.  (\ref{poln}) with control
(\ref{k12}) takes the following form
$$(A_0+K+R(t,x))x+B_0u(x)=$$ $$=(A_0+K+R(t,x))x-\left(\frac{1}{2}\;
B_0B_0^*D(\Theta(x))FD(\Theta(x))+B_0B_0^*K\right)x.$$ Due to the
fact that $B_0B_0^*K=K,$ the last equation takes the form
$$(A_0+K+R(t,x))x+B_0u(x)=(A_0+R(t,x))x-\frac{1}{2}\;
B_0B_0^*D(\Theta(x))FD(\Theta(x)x.$$ Put $y(\Theta,x)=D(\Theta)x.$
Then Eq.  (\ref{k10}) takes the following form
 \begin{equation}\label{f_2}
2a_0\Theta=(Fy(\Theta,x),y(\Theta,x)).
\end{equation}
Let us set $$
H=\textrm{diag}(H_1,\ldots,H_r),\quad\mbox{where}\quad
H_i=\textrm{diag}\left(-\frac{2n_i-2j+1}{2}\right)_{j=1}^{n_i} $$
and
\begin{equation}\label{deff1}
F^1=F-FH-HF=((2n-i-j+2)f_{ij})_{i,j=1}^n.
\end{equation}
If the matrix $F$ is positive defined, then Eq.  (\ref{f_2}) has a
unique positive solution $\Theta=\Theta(y)$ \cite[p. 108]{Kor2}.
 Since the controllability function is the time of motion, then the matrix
 $F^1$ is positive defined \cite[p. 106]{Kor2}. Let the constant $a_0$ satisfies inequality (\ref{k11}). Let us
investigate the closed-loop system (\ref{poln}) with control given
by relation (\ref{k12}). Let us denote the trajectory of this
system by $x(t)$ and let us find the derivative with respect to
the system $\dot\Theta=\frac{d}{dt}\Theta(x(t))$. From Eq.
(\ref{f_2})  it follows that
 \begin{equation}\label{f_3}
2a_0\dot{\Theta}=(F\dot{y}(\Theta,x),y(\Theta,x))+(Fy
(\Theta,x),\dot{y}(\Theta,x)).
\end{equation}

Let us find $\dot{y}(\Theta,x)$. We obtain that
$\displaystyle\frac{d}{d\Theta}D(\Theta)=\frac{1}{\Theta}H
D(\Theta).$ Therefore, $$\dot y(\Theta,x)=\dot D(\Theta
)x+D(\Theta )\dot x= \frac{\dot\Theta
}{\Theta}Hy(\Theta,x)+D(\Theta)A_0D^{-1}(\Theta)y(\Theta,x)+$$
$$+D(\Theta)R(t,x)D^{-1}(\Theta)y(\Theta,x)-
\frac{1}{2}D(\Theta)B_0B_0^*D(\Theta)Fy(\Theta,x).$$ Let us set
\begin{equation}\label{defs}
S(\Theta,t,x)=\Theta(FD(\Theta)R(t,x)D^{-1}(\Theta)+D^{-1}(\Theta)R^*(t,x)D(\Theta)F).
\end{equation} In
\cite[p. 109]{Kor2} it was proved that
$$D(\Theta)A_0D^{-1}(\Theta)=\Theta^{-1}A_0,\quad
D(\Theta)b_0=\Theta^{-1/2}b_0,\quad FA_0+A_0^*F-FB_0B_0^*F=-F^1.$$
From (\ref{f_3}) we see that
$$\dot{\Theta}(2a_0-\frac{1}{\Theta}((FH+HF)y(\Theta,x),
y(\Theta,x)))=\frac{1}{\Theta}((-F^1+S(\Theta,t,x))
y(\Theta,x),y(\Theta,x)).$$ Taking into account Eq. (\ref{f_2}),
we obtain that the derivative of the controllability function with
respect to system (\ref{poln}) is of the form:
\begin{equation}\label{f_4}
\dot{\Theta}=-1+\frac{(S(\Theta,t,x) y(\Theta,x),y(\Theta,x))}
{(F^1 y(\Theta,x),y(\Theta,x))}.
\end{equation}

Let us introduce the following notation:\\
$\bullet$ $M^*$ is the transpose matrix to the matrix $M;$\\
$\bullet$ $\sigma(M)$  is the spectrum of matrix  $M;$\\
$\bullet$ $\lambda_{min}(M)=\min\{\lambda:\;\lambda\in
\sigma(M)\};$\\
$\bullet$ $\lambda_{max}(M)=\max\limits\{\lambda:\;\lambda\in
\sigma(M)\};$\\
$\bullet$ $\rho(M)=\max\{|\lambda|,\;\lambda\in
\sigma(M)\}$ is spectral radius of matrix $M;$\\
$\bullet$ $|M|=(|m_{ij}|)_{i,j=1}^n$  is the absolute value of
matrix $M,$ i. e. matrix which consists of
 absolute values of the  elements of matrix $M;$\\
$\bullet$ $\widetilde{G}=|{(F^1)}^{-1}|\cdot(F
\widetilde{R}+\widetilde{R}^*F),$ where the matrix $\widetilde{R}$
coincides with the  matrix $R(t,x)$ at $r_{mj}(t,x)=1.$

Let us set $y=y(\Theta,x).$ Let us find the exact estimate for
 $\dot\Theta.$ To this end we find the largest and
  smallest values of the ratio
$(S(\Theta,t,x)y,y)/(F^1 y,y)$ at $y\not=0.$  Let us consider the
problem
$$(S(\Theta,t,x)y,y)\rightarrow extr,\quad y\in\{y:\;(F^1 y,y)=c\}.$$
We solve this problem using the method of Lagrange multipliers.
The Lagrange function takes the form
$${\cal L}(y,\lambda)=(S(\Theta,t,x)y,y)-\lambda[(F^1 y,y)-c].$$
From the necessary condition of the extremum we obtain
that\linebreak $S(\Theta,t,x)y-\lambda F^1 y=0.$ So at the
extremum point the following condition holds:
$(S(\Theta,t,x)y,y)=\lambda(F^1 y,y),$ moreover
$\lambda\in\sigma({(F^1)}^{-1}S(\Theta,t,x)).$ Therefore,
$${\lambda}_{min}({(F^1)}^{-1}S(\Theta,t,x))\leq\frac{(S(\Theta,t,x)y,y)}
{(F^1 y,y)}\le {\lambda}_{max}({(F^1)}^{-1}S(\Theta,t,x)).$$ Thus,
from(\ref{f_4}) we obtain that
\begin{equation}\label{f_5}
\dot{\Theta}\leq-1+{\lambda}_{max}({(F^1)}^{-1}S(\Theta,t,x)).
\end{equation}

\subsection{Perturbations of the superdiagonal
elements} Suppose that the  $(n_i\times n_i)$ matrices $R_i(t,x)$
have nonzero elements only at the main superdiagonal and
$\hat{R}(t,x)\equiv 0.$ Then system (\ref{f_1}) has the following
form:
\begin{equation}\label{f_1n} \left\{\begin{array}{l}
\dot{x}_{s_{i-1}+j}=(1+r_{(s_{i-1}+j)(s_{i-1}+j+1)}(t,x))x_{s_{i-1}+j+1},\;j=1,\ldots,n_i-1,\\
\dot{x}_{s_i}=\sum\limits_{j=1}^nk_{s_ij}x_j +u_i,\;i=1,\ldots,r.
\end{array}\right.
\end{equation}

Similarly to \cite[p. 109]{Kor2}, one can show that
$D(\Theta)R(t,x)D^{-1}(\Theta)=\Theta^{-1}R(t,x)$  (by using the
fact that in the case under consideration the matrix $R(t,x)$ has
the same structure as $A_0$). So we obtain that
\begin{equation}\label{defs0}S(\Theta,t,x)=S_0(t,x)=FR(t,x)+R^*(t,x)F.\end{equation}
It should be noted that the matrix $S_0(t,x)$ does not depend on
$\Theta$. This observation is crucial for our method of solving
the robust feedback synthesis. Indeed, the explicit form of
$S_0(t,x)$ is $S_0(t,x)=\textrm{diag}\left(S_{1}(t,x), \ldots,
S_{r}(t,x)\right),\;$ where $S_i(t,x)=$
$$\left(\begin{array}{cccc}
0&f_{11}r_{12}&\ldots&f_{1(n_i-1)}r_{(n_i-1)n_i}\\
f_{11}r_{12}&2f_{12}r_{12}&\ldots&f_{1n_i}r_{12}+f_{2(n_i-1)}r_{(n_i-1)n_i}\\
f_{12}r_{23}&f_{13}r_{12}+f_{22}r_{23}&\ldots&f_{2n_i}r_{23}+f_{3(n_i-1)}r_{(n_i-1)n_i}\\
&\ldots&&\\
f_{1(n_i-1)}r_{(n_i-1)n_i}&f_{1n_i}r_{12}+f_{2(n_i-1)}r_{(n_i-1)n_i}&\ldots&2f_{(n_i-1)n_i}r_{(n_i-1)n_i}
\end{array}\right),$$
and $r_{mj}=r_{mj}(t,x).$

\begin{theorem}\label{thglob} Let
 $\gamma$ be an arbitrary number which satisfies the
inequality $0<\gamma<1.$ Let
\begin{equation}\label{f_7new}
\Delta=\frac{(1-\gamma)}{\rho (\widetilde{G})}.
\end{equation}
Let the controllability function $\Theta=\Theta(x),$ $ x\neq 0,$
be a unique positive solution of Eq.  (\ref{k10}), where the
constant $a_0$ satisfies  inequality (\ref{k11}).

Then at the domain $Q$ specified by the equality
$\;Q=\{x:\;\Theta(x)\leq c\}$ the control  given by relation
(\ref{k12}) solves the local robust feedback synthesis for system
(\ref{f_1n}). Moreover,  the trajectory $x(t)$ of the closed-loop
system (\ref{poln}), starting at an arbitrary initial point
$x(0)=x_0\in Q,$ ends at the origin at some finite time
$T(x_0,\mathcal{R})$ satisfying the estimate
\begin{equation}\label{oc}
T(x_0,\mathcal{R})\leq \frac{\Theta(x_0)}{\gamma}.\end{equation}

In the case when $K\equiv 0,$ the robust feedback synthesis is
global.
\end{theorem}

\begin{proof}  Since $B_0=\textrm{diag} \left(B_{01}, \ldots, B_{0r}\right),$ then the matrices $A_0$ and
$B_0$ have a block structure. So the matrix $F^{-1}$ given by
(\ref{deff}) is of the form
$$F^{-1}=\textrm{diag}(F^{-1}_1,\ldots,F^{-1}_r),$$
where (see \cite[p. 98]{Kor2})
\begin{equation}\label{defsm}
\begin{array}{c} F^{-1}_i=\int\limits_0^1
(1-t)e^{-A_{0i}t}B_{0i}B_{0i}^{*}e^{-A_{0i}^{*}t}dt=\\
\left(\displaystyle\frac{(-1)^{m+j}}{(n_i-m)!(n_i-j)!(2n_i-m-j+1)(2n_i-m-j+2)}\right)_{m,j=1}^{n_i}.
\end{array}
\end{equation}

Let us fix value of $i$ and consider the matrix $F_i$ which is
inverse to the matrix $F^{-1}_i.$ Let us prove that the elements
of the matrix $F_i$ are positive. To this end we analyze  the
matrix
$$\widetilde{M}=\left(\displaystyle\frac{1}{(2n_i-m-j+1)(2n_i-m-j+2)}\right)_{m,j=1}^{n_i}.$$
Put $d_m=(-1)^{m}(n_i-m)!$ The elements of the matrix $F^{-1}_i$
can be calculated from the elements of the matrix $\widetilde{M}$
by multiplying every element of row $m$ by $d_m$ and every element
of column $j$ by $d_j.$ It is known that if every element of row
$m$ of  the matrix is  multiplied by $\varepsilon\neq 0,$ then
every element of column $m$ in the inverse matrix will be divided
by $\varepsilon.$ A similar assertion is true for the columns.
Then, in order to determine the elements of the matrix $F_{i}$ we
should divide every element of column $m$ of the matrix
$\widetilde{M}^{-1}$ by $d_m,$ and every element of row $j$ of the
matrix $\widetilde{M}^{-1}$ by $d_j.$ Therefore, the element with
the number $mj$ will be divided by $d_md_j,$ $\textrm{sign}\;
d_md_j=(-1)^{m+j}.$

Let us prove that all the minors of the matrix $\widetilde{M}$ are
positive. It is known that all the minors of $n_i\times n_i$
matrix $\widetilde{M}$  are positive if its $s$ order minors
composed from consecutive $s$ rows and consecutive $s$ columns are
positive \cite[Theorem 3.3]{karlin}. This theorem was first proved
 in \cite{Fekete}. So in the matrix $\widetilde{M}$ we consider only
submatrices composed from  consecutive $s$ rows  $\bar{r}+1,
\bar{r}+2, \ldots, \bar{r}+s$ and consecutive $s$ columns
$\bar{c}+1, \bar{c}+2, \ldots, \bar{c}+s.$  In addition, any such
submatrix  is the  Schur product of the Cauchy matrices. A Cauchy
matrix is a matrix of the form $\left(
\displaystyle\frac{1}{x_m+y_j}\right)_{m,j=1}^n$ \cite[Theorem
1.2.12.1]{prasolov}. Each consecutive submatrix of the matrices
$\left(\displaystyle\frac{1}{2n_i-m-j+1}\right)_{m,j=1}^{n_i}$ and
$\left(\displaystyle\frac{1}{2n_i-m-j+2}\right)_{m,j=1}^{n_i}$ is
a Cauchy matrix (put for the first matrix
$x_m=n_i-m,\;y_j=n_i-j+1$). The determinant of the Cauchy matrix
is determined in \cite[Theorem 1.2.12.1]{prasolov} by the formula
\begin{equation}\label{koshi}
\displaystyle\frac{\prod\limits_{m>j}(x_m-x_j)(y_m-y_j)}{\prod\limits_{m,j}(x_m+y_j)}.
\end{equation}
Each consecutive submatrix of the matrices
$\left(\displaystyle\frac{1}{2n_i-m-j+1}\right)_{m,j=1}^{n_i}$ and
$\left(\displaystyle\frac{1}{2n_i-m-j+2}\right)_{m,j=1}^{n_i}$ is
a positive definite matrix due to the Silvester criteria and
formula (\ref{koshi}). The Schur product of the matrices
$\left(\displaystyle\frac{1}{2n_i-m-j+1}\right)_{m,j=1}^{n_i}$ and
$\left(\displaystyle\frac{1}{2n_i-m-j+2}\right)_{m,j=1}^{n_i}$
 is the matrix of the form
 $$\left(\displaystyle\frac{1}{(2n_i-m-j+1)(2n_i-m-j+2)}\right)_{m,j=1}^{n_i}$$
 and it is equal to $\widetilde{M}.$
The Schur product of positive definite matrices is a positive
definite matrix \cite[Theorem 6.4.2.1]{prasolov}. Hence, in the
matrix $\widetilde{M}$ all the submatrices composed from
consecutive $s$ rows  $\bar{r}+1, \bar{r}+2, \ldots, \bar{r}+s$
and consecutive $s$ columns  $\bar{c}+1, \bar{c}+2, \ldots,
\bar{c}+s$ are positive. Therefore, all minors of the matrix
$\widetilde{M}$ are positive. Then the minors of the order $n_i-1$
and $n_i,$ in particular, are also positive. Hence, the elements
of the matrix inverse to the matrix $\widetilde{M}$ have the sign
$(-1)^{m+j}.$ This implies that all the elements of the matrix
$F_i$ inverse to matrix $F^{-1}_i$ are positive.

It is known that ${\lambda}_{max}({(F^1)}^{-1}S_0(t,x))\leq \rho
({(F^1)}^{-1}S_0(t,x)).$ We claim that $\rho
({(F^1)}^{-1}S_0(t,x))\leq \rho |{(F^1)}^{-1}S_0(t,x)|.$  To prove
this inequality we need the following Theorem.
\begin{theorem}\cite[Theorem 8.1.18]{Xorn} Let $M$ and $N$ be some matrices.
Then\\1. $|M\cdot N|\leq |M|\cdot |N|;\quad$\\ 2. If $|M|\leq N,$
then $\rho(M)\leq \rho(|M|)\leq\rho(N).$
\end{theorem}

Therefore $\rho ({(F^1)}^{-1}S_0(t,x))\leq \rho
|{(F^1)}^{-1}S_0(t,x)|\leq \Delta \rho (\widetilde{G}).$ Here we
use the fact that the elements of the matrix $F$ are positive.
 Let us  substitute the last inequality into inequality (\ref{f_5}). We obtain that
\begin{equation}\label{f_6}
\dot{\Theta}\leq-1+\Delta \rho(\widetilde{G}).
\end{equation}

If we assume that $-1+\Delta \rho (\widetilde{G})\leq -\gamma,$
then $\dot{\Theta}\leq -\gamma.$ Similarly to \cite[Theorem
1.2]{Kor2}, the estimate on the time of motion (\ref{oc}) follows
from  the last inequality.

To complete the proof of the theorem, boundedness of the control
has to be established. Since
$B_0^*D(\Theta)=\Theta^{-\frac{1}{2}}B_0^*,$ the control given by
(\ref{k12}) can be rewritten in the form
$$u(x)=-\left(\frac{\Theta^{-\frac{1}{2}}}{2}\;B_0^*F+B_0^*KD^{-1}(\Theta(x))\right)y(\Theta,x)$$
Since $\|y(\Theta,x)\|^2\le 2a_0\Theta(x)\|F^{-1}\|$ and
$$
\|D^{-1}(\Theta(x))\|= \left\{\begin{array}{ll} \displaystyle
\Theta^\frac{1}{2}&\mbox{if}\;\;0<\Theta<1,\\
\displaystyle \Theta^\frac{2n_1-1}{2}&\mbox{if}\;\;\Theta\ge 1,
\end{array}\right.$$
at $\Theta(x)\leq c$ we get
$$ \|u (x)\|\leq \left(\frac{1}{2}\|B_0^*F\|+\max\{c^{n_1},c\}\|B_0^*K\|\right)
\sqrt{2a_0\|F ^{-1}\|}.$$ Let the constant $a_0$ satisfy
inequality (\ref{k11}). Then from  the last inequality we obtain
that $\;\|u (x)\|\le 1\;$ for all $x\in Q.$ Due to \cite[Theorem
2.3]{Kor2} the control $u(x)$ of the form (\ref{k12}) solves the
local feedback synthesis for system (\ref{f_1n}). The proof of
theorem is completed.
\end{proof}

\subsection{The general case} Let the matrix $R(t,x)$ has the form
given in (\ref{newr}). Then the elements of $S(\Theta,t,x)$
defined by relation (\ref{defs}) are  polynomials in $\Theta$
whose degree does not exceed $n_1.$ Reasoning similarly to the
case with the perturbations of the superdiagonal elements, from
inequality (\ref{f_5}) it follows that
$$
\dot{\Theta}\leq-1+\rho ({(F^1)}^{-1}S(\Theta,t,x))\leq-1+\Delta
\max\{c^{n_1},c\}\cdot\rho (\widetilde{G})
$$
at $\Theta(x)\leq c.$ If we assume that
\begin{equation}\label{dottet}
 -1+\Delta
\max\{c^{n_1},c\}\cdot\rho (\widetilde{G})\leq -\gamma,
\end{equation}
 then $\dot{\Theta}\leq -\gamma.$ Thus, the
following theorem is valid.

\begin{theorem}\label{thloc}
Let the controllability function $\Theta=\Theta(x),$ $ x\neq 0,$
be a unique positive solution of Eq.  (\ref{k10}), where the
constant $a_0$ satisfies  inequality (\ref{k11}). Let the
solvability domain be defined by $\;Q=\{x:\;\Theta(x)\leq c\},$
where $Q$ is ellipsoid. Let  $\gamma$ be an arbitrary number which
satisfies the inequality\linebreak $0<\gamma<1.$ Let
\begin{equation}\label{f_7}
\Delta=\frac{(1-\gamma)}{\max\{c^{n_1},c\}\cdot\rho
(\widetilde{G})}.
\end{equation}

Then in the domain $Q$ the control given by  relation (\ref{k12})
solves the local robust feedback synthesis for system (\ref{f_1}).
Moreover, the trajectory $x(t)$ of the closed-loop system
(\ref{poln}), starting at an arbitrary initial point $x(0)=x_0\in
Q,$ ends at the origin at some finite time $T(x_0,\mathcal{R}),$
where the time of motion $T(x_0,\mathcal{R})$ satisfies inequality
(\ref{oc}).
\end{theorem}

\begin{remark}
{\rm If we solve inequality (\ref{dottet}) with respect to $c$ and
consider $\Delta$ to be arbitrary, then we obtain the following
solvability domain of the synthesis problem:
$\;Q=\{x:\;\Theta(x)\leq c\}.$}
\end{remark}

\begin{remark}
{\rm Value of $\Delta$ is monotonically decreasing in $\gamma.$ In
addition, the inequality for the time of motion
$T(x_0,\mathcal{R})$ given by (\ref{oc}) is also monotonically
decreasing in $\gamma.$ The value $\Delta\rightarrow\max$ at
$\gamma\rightarrow 0.$ Moreover $T(x_0,\mathcal{R})\rightarrow
+\infty$  at $\Delta\rightarrow 0.$}
\end{remark}

\begin{remark}
{\rm Let $R(t,x)\in\mathcal{R}.$ To determine the trajectory
starting at a given initial point $x_0\in Q$ we act as follows. We
 solve Eq.  (\ref{k10}) at $x=x_0$ and
find its unique positive root $\Theta(x_0)=\Theta_0.$ Put
$\theta(t)=\Theta(x(t)).$ The trajectory satisfies the following
system:
\begin{equation}\label{Coshi}
\left\{\begin{array}{l}
\dot{x}=(A_0+R(t,x))x-\frac{1}{2}\;B_0^*D(\theta(x))FD(\theta(x))x,\\[7pt]
\dot{\theta}=\displaystyle\frac{(-F^1+S(\Theta,t,x))
D(\theta)x,D(\theta)x)} {(F^1
D(\theta)x,D(\theta)x)},\\[10pt]
x(0)=x_0,\;\theta(0)=\Theta_0.
\end{array}\right.
\end{equation}
It should be noted that in order to determine $\Theta_0$ it
suffices to solve Eq.  (\ref{k10}) only once.}
\end{remark}

\subsection{Stopping the oscillations of the system of two
coupled pendulums} Let us consider a mechanical system which
consists of two pendulums coupled by a spring. Pendulums oscillate
in the same plane. We denote by $l_1$ and $l_2$ lengthes of
pendulums and $m_1$ and $m_2$ theirs masses.  The lengthes from
the suspension points  of two pendulums to the spring attachment
points are considered to be equal to each other and we denote them
by $h.$ The spring  stiffness is equal to $k.$  Oscillations of
this system without   a control were considered in many books (see
for ex. \cite[Sect. 6.1]{migulin}, \cite[Sect. 132]{strelkov}).

\begin{figure}[h!]
\begin{center}
\includegraphics[scale=0.5]{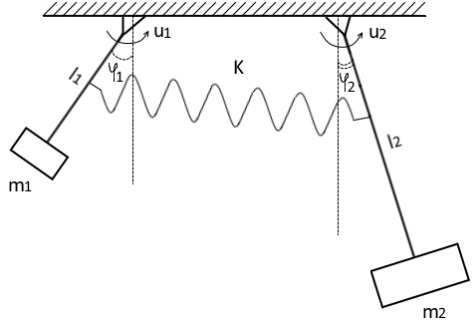}
\caption{A system which consist of two coupled
pendulums}\label{ris2}
\end{center}
\end{figure}

Let us consider the controllable motion of this system. Pairs of
forces $u_1$ and $u_2$ act as shown in Fig. \ref{ris2}. The
linearized equations of the motion of this pendulums are of the
form:
\begin{equation}\label{pr2}
\left\{\begin{array}{l}
\ddot{\varphi}_1=-\displaystyle\frac{m_1gl_1+kh^2}{m_1l_1^2}\;\varphi_1
+\frac{kh^2}{m_1l_1^2}\;\varphi_2+u_1,\\[10pt]
 \ddot{\varphi}_2=\displaystyle\frac{kh^2}{m_2l_2^2}\;\varphi_1-\frac{m_2gl_2+kh^2}{m_2l_2^2}\;\varphi_2+u_2.
\end{array}\right.
\end{equation}

Pairs of forces $u_1$ and $u_2$ satisfy the inequality
$\|(u_1,u_2)^*\|=\sqrt{u_1^2+u_2^2}\le1.$ We assume that positive
value of $u_i$ corresponds the case then moments of the force acts
in a clockwise direction. The force act tangentially to trajectory
of motion.

\textit{The first case.} Suppose that the values of
$m_1,$ $m_2,$ $l_1,$ $l_2$ and $h$ are
known. Suppose that the spring stiffness $k$  is unknown. Let us set
$$\displaystyle\frac{kh^2}{m_1l_1^2}=r_{21},\;\displaystyle\frac{kh^2}{m_2l_2^2}=r_{41},\;
\displaystyle\frac{g}{l_1}=k_{21},\;\displaystyle\frac{g}{l_2}=k_{43}.$$
By changing the variables
$$
x_1 = \varphi_1,\quad x_2 = \dot \varphi_1,\quad x_3 =
\varphi_2,\quad x_4 = \dot \varphi_2 $$ system (\ref{pr2}) is
reduced to the following form:
\begin{equation}\label{pr3n}
\left\{\begin{array}{l} \dot x_1 = x_2, \\
\dot x_2 = -(r_{21}+k_{21}) x_1 +r_{21}x_3+ u_1, \\
\dot x_3 = x_4, \\
\dot x_4 = r_{41}x_1 -(r_{41}+k_{43})x_3+ u_2.
\end{array}\right.
\end{equation}
The coefficients $r_{21}$ and $r_{41}$ are unknown constants.

System (\ref{pr3n}) can be written in the matrix form:
\begin{equation}\label{pr4}
\dot x = (A_0+K+R)x + B_0u ,
\end{equation}
where
\begin{equation}\label{defa}A_0=\left(
\begin{array}{cccc}
 0 & 1 & 0 & 0 \\
 0 & 0& 0 & 0 \\
 0 & 0 & 0 & 1 \\
0 & 0& 0 & 0 \\
\end{array}\right),\;
B_0=\left(
\begin{array}{cc}
 0 & 0 \\
 1 & 0 \\
 0 & 0 \\
 0 & 1
\end{array}
\right),
\end{equation}
$$ K=\left(
\begin{array}{cccc}
 0 & 0 & 0 & 0 \\
 -k_{21} & 0& 0 & 0 \\
 0 & 0 & 0 & 0 \\
0 & 0& -k_{43} & 0
\end{array}\right),\;
R=\left(
\begin{array}{cccc}
 0 & 0 & 0 & 0 \\
 -r_{21} & 0& r_{21} & 0 \\
 0 & 0 & 0 & 0 \\
r_{41} & 0& -r_{41} & 0
\end{array}\right),\;$$ and $n_1=2\;,n_2=2\;,s_1=2\;,s_2=n=4.$

Let us consider the robust feedback synthesis for system
(\ref{pr4}). Since for any fixed stiffness $k$ the following
equation holds: $rg (B_0,(A_0+K+R)B_0)=4,$ then this system is
completely controllable.

The matrices $F$ and $D(\Theta)$ given by relations (\ref{deff})
and (\ref{defd}) correspondingly are of the following form:
\begin{equation}\label{deffd} F= \left(\begin{array}{cccc} 36&12&0&0\\
12&6&0&0\\
0&0&36&12\\
0&0&12&6
\end{array}\right), \quad
D(\Theta)= \left(\begin{array}{cccc}
\Theta^{-\frac{3}{2}}&0&0&0\\[5pt]
0&\Theta^{-\frac{1}{2}}&0&0\\[5pt]
0&0&\Theta^{-\frac{3}{2}}&0\\[5pt]
0&0&0&\Theta^{-\frac{1}{2}}
\end{array}\right).
\end{equation}

Let $x=(x_1,x_2,x_3,x_4)\neq 0$ and determine the controllability
function $\Theta=\Theta(x)$ as a unique positive solution of Eq.
(\ref{k10}). In the analyzed case, this equation takes the form:
\begin{equation}\label{pr5}
2a_0\Theta^4=36x_1^2+24\Theta
x_1x_2+6\Theta^2x_2^2+36x_3^2+24\Theta x_3x_4+6\Theta^2x_4^2.
\end{equation}
At $x=0$ we put $\Theta(0)=0.$ We consider the solution of the
robust feedback synthesis in the ellipsoid
$\;Q=\{x:\;\Theta(x)\leq c\}.$ The constant $c>0$ is defined
below. The constant $a_0$ satisfies inequality (\ref{k11}) which
takes the form:
\begin{equation}\label{pr6n}
0<a_0\leq
\frac{3.58}{(13.42+2\max\{c^2,c\}\max\{k_{21},k_{43}\})^2}.
\end{equation}
In order to the solvability domain contains the ellipsoid of the
largest size, we choose $a_0$ as the largest value which satisfies
(\ref{pr6n}).

The control given by relation (\ref{k12}) which solves the robust
feedback synthesis is of the following form:
$$
u(x)=\left(\begin{array}{l}u_1(x)\\u_2(x)
\end{array}\right)=\left(\begin{array}{l}
-\displaystyle\frac{6x_1}{\Theta^2(x)}-\displaystyle\frac{3x_2}{\Theta(x)}+k_{21}x_1\\
-\displaystyle\frac{6x_3}{\Theta^2(x)}-\displaystyle\frac{3x_4}{\Theta(x)}+k_{43}x_3
\end{array}\right),
$$ where $\Theta=\Theta(x)$ is a unique positive solution of Eq.
(\ref{pr5}). For any value of $k$ this control steers an arbitrary
initial point $x_0$ to the origin in some finite time
$T(x_0,k)\leq \Theta(x_0)/\gamma,$ where $\gamma$ is an arbitrary
number which satisfies the inequality $0<\gamma<1.$

The matrix $S=S(\Theta,t,x)$ given by relation (\ref{defs}) is of
the form:
$$S(\Theta)= \left(\begin{array}{cccc}
-24r_{21}\Theta^2&-6r_{21}\Theta^2&12(r_{21}+r_{41})\Theta^2&6r_{41}\Theta^2\\
-6r_{21}\Theta^2&0&6r_{21}\Theta^2&0\\
12(r_{21}+r_{41})\Theta^2&6r_{21}\Theta^2&-24r_{41}\Theta^2&-6r_{41}\Theta^2\\
6r_{41}\Theta^2&0&-6r_{41}\Theta^2&0
\end{array}\right),
$$ where $\Theta=\Theta(x)$ is a unique positive solution of Eq.  (\ref{pr5}).

Let us find an estimate for the solvability domain. To this end we
find  $c$ from inequality (\ref{dottet}), which takes the form
\begin{equation}\label{pr8}
-1+\Delta \max\{c^2,c\}\cdot\rho (\widetilde{G})\leq -\gamma,
\end{equation}
where $\widetilde{G}= \left(\begin{array}{cccc}
 \displaystyle\frac{7}{6} & \displaystyle\frac{1}{6} & \displaystyle\frac{7}{6} & \displaystyle\frac{1}{6} \\[8pt]
 4 & \displaystyle\frac{1}{2}& 4 & \displaystyle\frac{1}{2} \\[8pt]
 \displaystyle\frac{7}{6}&\displaystyle\frac{1}{6}&\displaystyle\frac{7}{6} & \displaystyle\frac{1}{6}  \\[8pt]
4 & \displaystyle\frac{1}{2}&  4 & \displaystyle\frac{1}{2}
\end{array}\right),\;$
$\rho (\widetilde{G})\approx 8.4,\;$ $\Delta=k\cdot
\max\left\{\displaystyle\frac{h^2}{m_1l_1^2};\;\displaystyle\frac{h^2}{m_2l_2^2}\right\}.$\\
From (\ref{pr8}) it follows that
$$\max\{c^2,c\}\leq\displaystyle\frac{0.12(1-\gamma)}{r}=
\displaystyle\frac{0.12(1-\gamma)}{k
\max\left\{\displaystyle\frac{h^2}{m_1l_1^2};\;\displaystyle\frac{h^2}{m_2l_2^2}\right\}}.$$

Taking into account the inequality (\ref{f_5}), let us find a more
precise estimate for $c.$ At $x\in Q$ from (\ref{f_5}) it follows
that
$$\dot{\Theta}\leq-1+{\lambda}_{max}({(F^1)}^{-1}S(\Theta))=-1+\displaystyle\frac{\left(r_{21}+r_{41}+
2\sqrt{2(r_{21}^2+r_{41}^2)}\right) \Theta^2}{6}\leq$$
$$\leq-1+\displaystyle\frac{\left(r_{21}+r_{41}+
2\sqrt{2(r_{21}^2+r_{41}^2)}\right) c^2}{6}.$$ Let $c>0$ be such
that the following inequality holds:
\begin{equation}\label{pr9n}
-1+\displaystyle\frac{\left(r_{21}+r_{41}+
2\sqrt{2(r_{21}^2+r_{41}^2)}\right) c^2}{6}\leq -\gamma.\;
\end{equation}
Then $\dot{\Theta}\leq -\gamma.$ From (\ref{pr9n}) it follows that
$c\leq\sqrt{\frac{6(1-\gamma)}{(r_{21}+r_{41}+
2\sqrt{2(r_{21}^2+r_{41}^2)}}}.$ In order to solvability domain
contains the ellipsoid of the largest size, we choose $c$ as the
largest value which satisfies (\ref{pr9n}). So, we obtain the
following solvability domain:
\begin{equation}\label{pr10n}
Q=\left\{x:\;\Theta(x)\leq
\sqrt{\frac{6(1-\gamma)}{k\left(\frac{h^2}{m_1l_1^2}+\frac{h^2}{m_2l_2^2}+
2\sqrt{\frac{2h^4}{m_1^2l_1^4}+\frac{2h^4}{m_2^2l_2^4}}\right)}}\right\}.
\end{equation}

Let us consider the values of the parameters
$$m_1=1,\;m_2=2,\;l_1=60,\;l_2=30,\;h=7.5,\;\gamma=0.001.$$
Then $\displaystyle\frac{h}{l_1}=\displaystyle\frac{1}{8},\;
\displaystyle\frac{h}{l_2}=\displaystyle\frac{1}{4},\;$
$k_{21}=\displaystyle\frac{g}{l_1}\approx
0.16,\;k_{43}=\displaystyle\frac{g}{l_2}\approx
0.32,\;r_{21}=\displaystyle\frac{k}{64},\;r_{41}=\displaystyle\frac{k}{32}.$

Let the stiffness $k$ satisfies the constraint  $k\leq 4,$ but the
value of $k$ is unknown. Then the set of points (\ref{pr10n}) from
which we may steer to the origin  is the ellipsoid of the form
$Q=\{x:\;\Theta(x)\leq 3.2\}.$ Besides from (\ref{pr10n}) it
follows that the stiffness $k$ decreases as values of axes of
ellipsoid $Q$ increases. At $c=3.2$ inequality (\ref{pr6n}) on
$a_0$ takes the form: $a_0\leq 0.0088\ldots$ Put $a_0=0.0088.$

Let the initial point be equal to $x(0)=(-0.3,0.3,0,0),\;x(0)\in
Q.$ The unique positive solution $\Theta_0$ of Eq. (\ref{pr5})
$\Theta_0\approx3.2$. Let $x=x(t,k_0)$ be the trajectory of system
(\ref{Coshi}), which is realized at some coefficient of stiffness
$k_0$ which satisfies inequality $k_0\leq 4.$ Put
$\theta(t)=\Theta(x(t,k_0)).$ The trajectory $x=x(t,k_0)$
satisfies the following  system:
\begin{equation}\label{pr_11n}
\left\{\begin{array}{l} \dot{x}_1=x_2,\\[8pt]
\dot{x}_2=\displaystyle\frac{k_0}{64}(-x_1+x_3)-\displaystyle\frac{6\;x_1}{\theta^2}-\displaystyle\frac{3x_2}{\theta},\\[8pt]
\dot{x}_3=x_4,\\[8pt]
\dot{x}_2=\displaystyle\frac{k_0}{32}(x_1-x_3)-\displaystyle\frac{6\;x_1}{\theta^2}-\displaystyle\frac{3x_2}{\theta},\\[8pt]
\dot{\theta}=\phi,\\
x_1(0)=-0.3,\quad x_2(0)=0.3,\quad x_2(0)=0,\quad
x_2(0)=0,\quad\theta(0)=3.2,
\end{array}\right.
\end{equation}
where
$$\;\begin{array}{ll}
\phi=&-((12+0.03\;k_0\;\theta^2)\;x_1^2+
(6+0.02\;k_0\;\theta^2)\;x_1\;x_2\;\theta+x_2^2\;\theta^2+\\
&+(12+0.06\;k_0\;\theta^2)\;x_3^2+
(6+0.03\;k_0\;\theta^2)\;x_3\;x_4\;\theta+x_4^2\;\theta^2-\\
&-0.09\;k_0\;x_1\;x_3\;\theta^2
-0.02\;k_0\;x_2\;x_3\;\theta^3-0.03\;k_0\;x_1\;x_4\;\theta^3)/\\
&/(12\;x_1^2+6x_1\;x_2\;\theta+x_2^2\;\theta^2+12\;x_3^2+6x_3\;x_4\;\theta+x_4^2\;\theta^2).
\end{array}$$

The  two-dimensional projection of domain $Q$ on the plane
$Ox_1x_2=O\varphi_1\dot{\varphi}_1$ or equally
$\bar{Q}=\{(x_1^0,x_2^0,0,0):\;\Theta(x_1^0,x_2^0,0,0)\leq 3.2\}$
is given in Fig.~\ref{ris5}. Let
$(x_1^0(t),x_2^0(t),x_3^0(t),x_4^0(t),\theta(t))$ be the solution
of system (\ref{pr_11n}) at $k_0=4.$ The curve
$(x_1^0(t),x_2^0(t))$ is also given in Fig. \ref{ris5} (the solid
 line). The curve $(\bar{x}_1^0(t),\bar{x}_2^0(t))$ (the dashed
 line), which corresponds to the case $k_0=0$ is also
given in Fig. \ref{ris5}.
\begin{figure}[h]
\begin{center}
\includegraphics[scale=0.6]{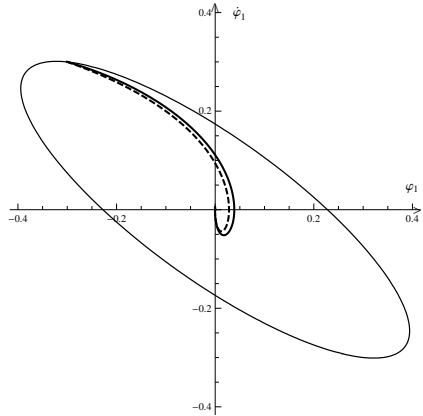} \caption{The
projection of the phase trajectory and the ellipsoid $Q$  on the
plane $O\varphi_1\dot{\varphi}_1$}\label{ris5}
\end{center}
\end{figure}
All the other trajectories fill up the domain between the
trajectories corresponding to $k_0=0$ and $k_0=4$ if stiffness
$k_0 $ satisfies the inequality $0\leq k_0\leq 4$ and trajectories
begin from $x(0).$ At $k_0=0$ the trajectory may be found from the
following system:
$$\left\{\begin{array}{l} \dot{x}_1=x_2,\quad
\dot{x}_2=-\displaystyle\frac{6\;x_1}{\theta^2}-\displaystyle\frac{3x_2}{\theta},\quad
\dot{x}_3=x_4,\quad
\dot{x}_2=-\displaystyle\frac{6\;x_1}{\theta^2}-\displaystyle\frac{3x_2}{\theta},\quad
\dot{\theta}=-1,\\
x_1(0)=-0.3,\quad x_2(0)=0.3,\quad x_2(0)=0,\quad
x_2(0)=0,\quad\theta(0)=3.2.
\end{array}\right.
$$
The plot of the components of the control on the trajectory
$$u_1=u_1(x_1^0(t),x_2^0(t),x_3^0(t),x_4^0(t))=
-\displaystyle\frac{6x_1^0(t)}{\theta^2(t)}-\displaystyle\frac{3x_2^0(t)}{\theta(t)}+
0.16x_1^0(t),$$
$$u_2=u_2(x_1^0(t),x_2^0(t),x_3^0(t),x_4^0(t))=
-\displaystyle\frac{6x_3^0(t)}{\theta^2(t)}-\displaystyle\frac{3x_4^0(t)}{\theta(t)}+
0.32x_3^0(t)$$ are given in Fig. \ref{ris6}. The norm of the
control $\|(u_1,u_2)^*\|=\sqrt{u_1^2+u_2^2}$ is given in
Fig.~\ref{ris7}, and we can see that  $\|(u_1,u_2)^*\|\le1$. The
controllability function $\theta(t)$ shown in in Fig.~\ref{ris8}
is close to the linear ($y=3.2-t$). The derivative of the
controllability function with respect to the system is given in
Fig.~\ref{ris9}, and we can see that it is negative. The estimate
for the time of motion (\ref{oc}) is of the form: $T\leq 3206.$ It
is fulfilled at all $0\leq k_0\leq 4,$ but at a particular value
of  $k_0$ the value of $T$ is less than 3206. The results of the
numerical calculations demonstrate that the time of motion $T$
from the point $x(0)$ at $k_0=4$ is $T\approx3.43,$
 besides it can be shown numerically that at
$0\leq k_0\leq 4$ the following inequality holds: $3.2\leq T\leq
3.43.$ All graphs are given at the trajectory at $k_0=4.$ At the
other values of $k_0$ graphs are similarly to that for present at Fig.
\ref{ris5} - \ref{ris9}.

\begin{figure}[h!]
\begin{center}
\begin{minipage}[h]{0.45\linewidth}
\includegraphics[width=1\linewidth]{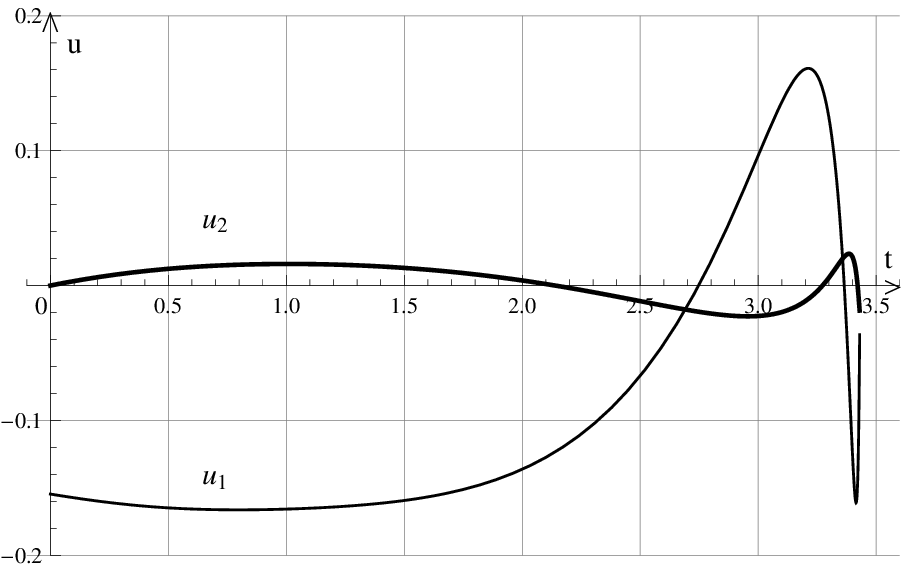}
\caption{The components of the control} \label{ris6}
\end{minipage}
\hfill
\begin{minipage}[h]{0.45\linewidth}
\includegraphics[width=1\linewidth]{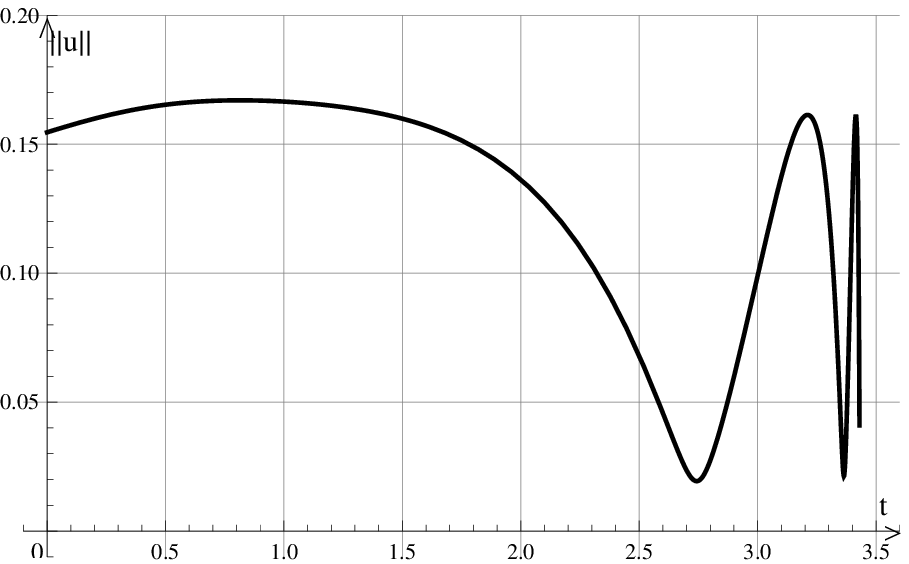}
\caption{The norm of the control}\label{ris7}
\end{minipage}
\end{center}
\end{figure}

\begin{figure}[h!]
\begin{center}
\begin{minipage}[h]{0.45\linewidth}
\includegraphics[width=1\linewidth]{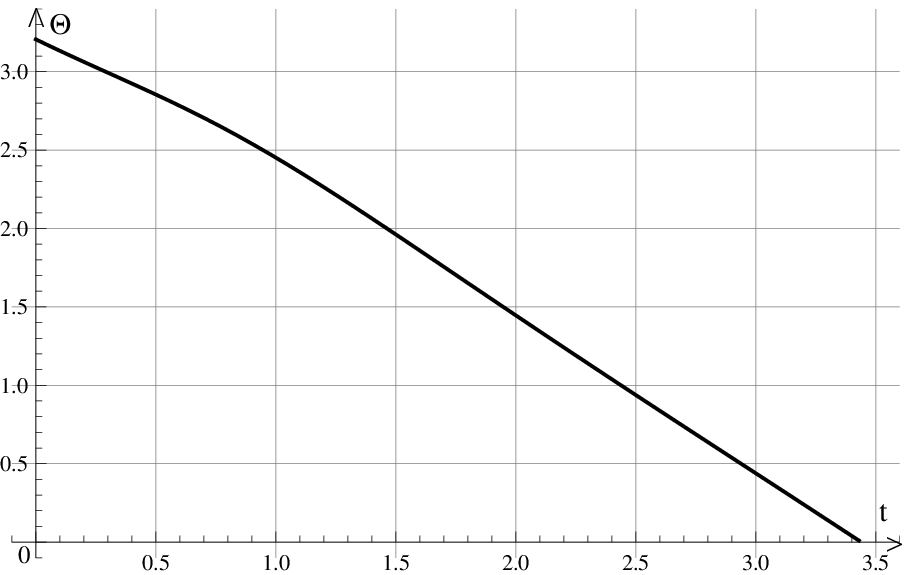}
\caption{The controllability function} \label{ris8}
\end{minipage}
\hfill
\begin{minipage}[h]{0.45\linewidth}
\includegraphics[width=1\linewidth]{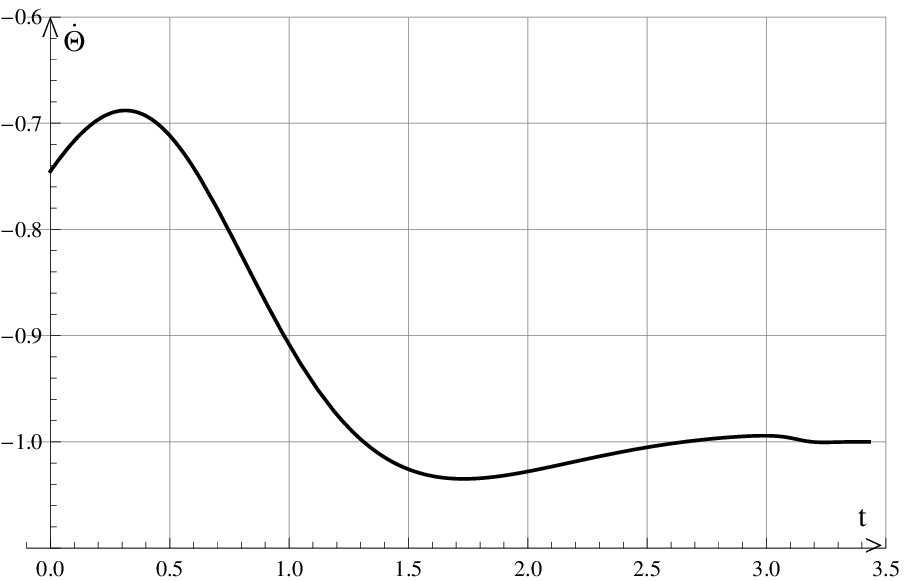}
\caption{The derivative of the controllability function w. r. t.
system}\label{ris9}
\end{minipage}
\end{center}
\end{figure}

\textit{The second case.} Let $l_1=l_2=l.$ Let us consider that the values
$m_1,$ $m_2$ and $k$ are known. Also we consider that the pendulum length
$l$ is unknown. Besides, the ratio $\displaystyle\frac{h}{l}$ is
known. Let us set
$$\displaystyle\frac{kh^2}{m_1l_1^2}=k_{21},\;\displaystyle\frac{kh^2}{m_2l_2^2}=k_{41},\;
\displaystyle\frac{g}{l}=r_{21}.$$

By  changing the  variables
$$
x_1 = \varphi_1,\quad x_2 = \dot \varphi_1,\quad x_3 =
\varphi_2,\quad x_4 = \dot \varphi_2 $$ system (\ref{pr2}) is
reduced to the following form:
$$
\left\{\begin{array}{l} \dot x_1 = x_2, \\
\dot x_2 = -(r_{21}+k_{21}) x_1 +k_{21}x_3+ u_1, \\
\dot x_3 = x_4, \\
\dot x_4 = k_{41}x_1 -(r_{21}+k_{41})x_3+ u_2.
\end{array}\right.
$$
The coefficient $r_{21}$ is unknown constant.

This system can be written in the matrix form (\ref{pr4}) where
the matrices $A_0$ and $B_0$ given by  relations (\ref{defa}) and
the matrices $K$ and $R$ are of the form:
$$K=\left(
\begin{array}{cccc}
 0 & 0 & 0 & 0 \\
 -k_{21} & 0& k_{21} & 0 \\
 0 & 0 & 0 & 0 \\
k_{41} & 0& -k_{41} & 0
\end{array}\right),\;
R=\left(
\begin{array}{cccc}
 0 & 0 & 0 & 0 \\
 -r_{21} & 0& 0 & 0 \\
 0 & 0 & 0 & 0 \\
0 & 0& -r_{21} & 0
\end{array}\right).\;$$

The matrices $F$ and $D(\Theta)$ are given by relations
(\ref{deffd}). Let us define  the controllability function
$\Theta=\Theta(x)$ at $x\neq 0$ as a unique positive solution of
Eq.  (\ref{pr5}). At $x=0$ we put $\Theta(0)=0.$ Similarly to the
first case we consider the solution of the robust feedback
synthesis in the ellipsoid $\;Q=\{x:\;\Theta(x)\leq c\}.$ The
constant $c>0$ is defined below. The constant $a_0$ satisfies
inequality (\ref{k11}) that takes the form:
\begin{equation}\label{pr6}
0<a_0\leq
\frac{3.58}{(13.42+2.83\max\{c^2,c\}\sqrt{k_{21}^2+k_{41}^2})^2}.
\end{equation}
In order to the solvability domain contains the ellipsoid of the
largest size, we choose $a_0$ as the largest value which satisfies
(\ref{pr6}).

The control given by relation (\ref{k12}) which solves the robust
feedback synthesis  is of the following form:
$$
u(x)=\left(\begin{array}{l}u_1(x)\\u_2(x)
\end{array}\right)=\left(\begin{array}{l}
-\displaystyle\frac{6x_1}{\Theta^2(x)}-\displaystyle\frac{3x_2}{\Theta(x)}+k_{21}(x_1-x_3)\\
-\displaystyle\frac{6x_3}{\Theta^2(x)}-\displaystyle\frac{3x_4}{\Theta(x)}+k_{41}(-x_1+x_3)
\end{array}\right),
$$ where $\Theta=\Theta(x)$ is a unique positive
solution of Eq.  (\ref{pr5}). For any value of $l$ this control
steers an arbitrary initial point $x_0$ to the origin in some
finite time  $T(x_0,l)\leq \Theta(x_0)/\gamma,$ where $\gamma$ is
an arbitrary number which satisfies the inequality $0<\gamma<1.$

The matrix $S=S(\Theta,t,x)$  given by relation (\ref{defs}) has
the following form:
$$S= \left(\begin{array}{cccc}
-\displaystyle\frac{24 g \Theta^2}{l_1}&-\displaystyle\frac{6 g \Theta^2}{l_1}&0&0\\[8pt]
-\displaystyle\frac{6 g \Theta^2}{l_1}&0&0&0\\[8pt]
0&0&-\displaystyle\frac{24 g \Theta^1}{l_1}&-\displaystyle\frac{6 g \Theta^2}{l_1}\\[8pt]
0&0&-\displaystyle\frac{6 g \Theta^2}{l_1}&0
\end{array}\right),
$$ where $\Theta=\Theta(x)$ is a unique positive solution of Eq.  (\ref{pr5}).

Taking into account the inequality (\ref{f_5}), let us find an
exact estimate for $c.$ Since
$\lambda_{max}((F^1)^{-1}S(\Theta))=\displaystyle\frac{g
\Theta^2}{2l},$ then at $x\in Q$ from (\ref{f_5}) it follows that
$$\dot{\Theta}\leq-1+{\lambda}_{max}({(F^1)}^{-1}S(\Theta))=-1+\displaystyle\frac{g
\Theta^2}{2l}\leq-1+\displaystyle\frac{g c^2}{2l}.$$ Let $c>0$ be
such that the following inequality holds:
\begin{equation}\label{pr9}
-1+\displaystyle\frac{g c^2}{2l}\leq -\gamma.\;
\end{equation}
Then $\dot{\Theta}\leq -\gamma.$ From (\ref{pr9}) it follows that
$c\leq\sqrt{0.2\;l(1-\gamma)}.$ In order to solvability domain
contain the ellipsoid of the largest size, we choose $c$ as the
largest value which satisfies (\ref{pr9}). So, we obtain the
following solvability domain:
\begin{equation}\label{pr10}
Q=\{x:\;\Theta(x)\leq \sqrt{0.2\;l(1-\gamma)}\}.
\end{equation}

Let
$$m_1=1,\;m_2=2,\;k=1,\;\displaystyle\frac{h}{l}=\displaystyle\frac{1}{4},\;\gamma=0.001.$$
Then
$k_{21}=\displaystyle\frac{kh^2}{m_1l^2}=\displaystyle\frac{1}{16},
\;k_{41}=\displaystyle\frac{kh^2}{m_2l^2}=\displaystyle\frac{1}{32},
\;r_{21}=\displaystyle\frac{9.8}{l}.$

Let the length $l$ satisfies the constraint $l\geq 30,$ but the
value of $l$ is unknown. Then the set of points (\ref{pr10}) from
which we may steer to the origin  is the ellipsoid of the form
$Q=\{x:\;\Theta(x)\leq 2.47\}.$ Besides from (\ref{pr10}) it
follows that the length $l$ decreases as values of axes of
ellipsoid $Q$ decrease. At $c=2.47$ inequality (\ref{pr6}) on
$a_0$ takes the form: $a_0\leq 0.016\ldots$ Put $a_0=0.016.$

Similarly  to the first case let the initial point be equal to
$x(0)=(-0.3,0.3,0,0),$ $\;x(0)\in Q.$ The unique positive solution
$\Theta_0$ of Eq.  (\ref{pr5}) is $\Theta_0\approx 2.44.$ The
estimate for the time of motion (\ref{oc}) is of the form: $T\leq
2438.$  It is fulfilled  at all $l\geq 30,$  but at a particular
value of $l$ the value of $T$ is  less than 2438. The results of
the numerical calculations demonstrate that the time of motion $T$
from the point $x(0)$ at $l=30$ is $T\approx3,$
 besides it can be shown numerically that at
$l\geq 30$ the following inequality holds: $2.44\leq T\leq 3.$ The
further considerations are similar to those in the first case.

\end{document}